\documentclass[a4paper]{article}

\usepackage{amsmath}
\usepackage{amsfonts}
\usepackage{amsthm}
\usepackage[all]{xy}
\SelectTips{cm}{10}
\usepackage{url}

\usepackage{a4wide}

\usepackage{titlesec}

\titleformat{\section}[block]
{\normalfont\Large\filcenter\bfseries}{\thesection.}{.33em}{}
\titlespacing*{\section}
{0pt}{3.5ex plus 1ex minus .2ex}{2.3ex plus .2ex}

\titleformat{\subsection}[runin]
{\normalfont\normalsize\bfseries}{\thesubsection.}{.33em}{}[.]
\titlespacing*{\subsection}
{0pt}{3.25ex plus 1ex minus .2ex}{.5em}

\theoremstyle{plain}
\newtheorem{theorem}{Theorem}
\newtheorem{lemma}[theorem]{Lemma}

\theoremstyle{definition}
\newtheorem{example}[theorem]{Example}

\newcommand{\bbR}{\mathbb R}
\newcommand{\bbZ}{\mathbb Z}

\newcommand{\thed}{d}
\newcommand{\theD}{D}
\newcommand{\then}{n}
\newcommand{\thenpo}{{n+1}}

\newcommand{\thek}{k}
\newcommand{\them}{m}
\newcommand{\themz}{{n}}
\newcommand{\themo}{{n'}}
\newcommand{\themp}{{m'}}
\newcommand{\theq}{q}
\newcommand{\theqp}{q'}
\newcommand{\qni}[2]{q}
\newcommand{\pni}[2]{p}
\newcommand{\mni}[2]{m}
\newcommand{\kni}[2]{k}
\newcommand{\knist}[2]{k_*}
\newcommand{\Thetani}[2]{\Theta}
\newcommand{\Pn}{P_\then}

\newcommand{\pr}{\operatorname{pr}}

\newcommand{\cyclic}[1]{\bbZ/#1}

\newcommand{\kn}{k_\then}

\newcommand{\conn}{\operatorname{conn}}

\renewcommand{\ker}{\operatorname{ker}}
\newcommand{\stdsimp}[1]{\Delta^{#1}}
\newcommand{\horn}[2]{%
\mbox{$\xy
<0pt,-\the\fontdimen22\textfont2>;p+<.1em,0em>:
{\ar@{-}(0,0.1);(3,7)},
{\ar@{-}(3,7);(6,0.1)},
{\ar@{-}(3.2,7);(6.2,0.1)},
{\ar@{-}(3.4,7);(6.4,0.1)}
\endxy\;\!{}^{#1}_{#2}$}}

\newcommand{\xra}[1]{\xrightarrow{#1}}
\newcommand{\xlra}[1]{\xrightarrow{\ #1\ }}

\newdir{c}{{}*!/-5pt/@^{(}}
\newdir{ >}{{}*!/-5pt/@{>}}

\newcommand{\htpy}[1][]{\mathbin{\:\!\!\xymatrix@1@C=15pt{{}\ar@{~>}[r]^{#1} & {}}}}

\newcommand{\pbsize}{15pt}
\newcommand{\pboffset}{.5}
\newcommand{\xycorner}[3]{\save #2="a";#1;"a"**{}?(\pboffset);"a"**\dir{-};#3;"a"**{}?(\pboffset);"a"**\dir{-}\restore}
\newcommand{\pb}{\xycorner{[]+<\pbsize,0pt>}{[]+<\pbsize,-\pbsize>}{[]+<0pt,-\pbsize>}}

\newcommand{\xymatrixb}[2]{%
\xy*i\xybox{\xymatrix{#2}};p!DR*!DR\xybox{\xymatrix{#1 \\ #2}}\endxy}

\newcommand{\leftbox}[2]{{}\phantom{#1} \save []+L*+<.5pc>!!<0pt,\the\fontdimen22\textfont2>!L{#1#2} \restore}
\newcommand{\rightbox}[2]{{}\phantom{#2} \save []+R*+<.5pc>!!<0pt,\the\fontdimen22\textfont2>!R{#1#2} \restore}

\title{Decidability of the extension problem for maps into odd-dimensional spheres\thanks{%
The research was supported by the grant P201/11/0528 of the Czech Science Foundation (GA \v CR).\newline
2010 \emph{Mathematics Subject Classification}. Primary 55Q05; Secondary 55S35.\newline
\emph{Key words and phrases}. Homotopy class, computation, higher difference.
}}

\author
{Luk\'a\v{s} Vok\v{r}\'{\i}nek}

\begin{document}

\maketitle

\begin{abstract}
In a recent paper~\cite{hardness}, it was shown that the problem of existence of a continuous map $X \to Y$ extending a given map $A \to Y$ defined on a subspace $A \subseteq X$ is undecidable, even for $Y$ an even-dimensional sphere. In the present paper, we prove that the same problem for $Y$ an odd-dimensional sphere is decidable. More generally, the same holds for any $d$-connected target space $Y$ whose homotopy groups $\pi_k Y$ are finite for $k>2d$.
\end{abstract}

\section{Introduction}

The main object of study of the present paper is the \emph{extension problem}. Given spaces $X$, $Y$ and a map $f \colon A \to Y$ defined on a subspace $A \subseteq X$, it questions the existence of a continuous extension
\[\xymatrix@C=3pc{
	A \ar@{c->}[d]_-\iota \ar[r]^-f & Y \\
	X \ar@{-->}[ru]_-g
}\]
If $Y$ is allowed non-simply connected, this problem is undecidable by a simple reduction to the word problem in groups. Thus, we restrict ourselves to the situation of a \emph{simply connected} $Y$.

In~\cite{Steenrod}, Steenrod expressed a hope that the extendability problem would be algorithmically solvable. It was proved in~\cite{stable} that this is indeed the case if one restricts to a suitably ``stable'' situation, i.e.\ if $\dim X \leq 2\conn Y+1$. The algorithm of that paper depended on computations with \emph{abelian groups} of homotopy classes of maps that are not available unstably. Later, the authors showed in~\cite{hardness} that the previous positive result was very much the best possible: the extension problem with $\dim X > 2\conn Y+1$ is undecidable, even for such a simple target space as $S^{d+1}$ with $d+1$ even. This undecidability result has implications to other problems, namely, \cite{robust} shows the undecidability of the problem of existence of a robust zero of a given PL-map $K \to \bbR^{d+2}$, again for $d$ even.

It may thus come as a bit of a surprise that the last two problems with $d+1$ odd are decidable -- this is the content of Theorem~\ref{t:extend} below. It applies to $Y=S^{\thed+1}$, $\thed+1$ odd, since in this case, $\pi_\then S^{\thed+1}$ is finite for $\then>\thed+1$. Again, \cite{robust} implies the decidability of the problem of existence of a robust zero of a given PL-map $K \to \bbR^{d+2}$, $d$ odd.

\begin{theorem}\label{t:extend}
There exists an algorithm that, given a pair of finite simplicial sets $(X,A)$, a finite $\thed$-connected simplicial set $Y$, $\thed\geq 1$, with homotopy groups $\pi_\then Y$ finite for all $2\thed<\then<\dim X$ and a simplicial map $f\colon A\to Y$, decides the existence of a continuous extension $g\colon X\to Y$ of $f$.
\end{theorem}

We do not have any bounds on the running time of such an algorithm. In the light of the \#P-hardness of the computation of the homotopy group $\pi_kY$ when $k$ is a part of the input (in unary), see~\cite{hardness}, one should not expect that this algorithm is polynomial-time when the dimension of $X$ is not fixed. However, even if $\dim X$ is bounded, it seems that our algorithm will not have polynomial running time. Nevertheless, the contrast with the undecidability for even-dimensional spheres is huge.

In Section~\ref{s:fibrewise}, we briefly discuss an extension of Theorem~\ref{t:extend} to the fibrewise equivariant situation of~\cite{aslep}. In the special case $A=\emptyset$, such an extension implies the decidability of the problem of existence of a $\bbZ/2$-equivariant map $X \to S^{d+1}$ when $d+1$ is odd. The \emph{index} of $X$, denoted $\operatorname{ind}X$, is the smallest $d+1$ for which such an equivariant map $X \to S^{d+1}$ exists; it has many applications in geometry and combinatorics. Thus, with the equivariant version of Theorem~\ref{t:extend}, it is possible to narrow $\operatorname{ind}X$ down to two possible values.

\section{Sets with an action and mappings to abelian groups}

Let $S$ and $T$ be sets with a binary operation ${+}\colon S\times T\to S$ that has a right-sided zero $0\in T$, i.e.\ such that $x+0=x$. We use the bracketing convention $x+y+z=(x+y)+z$. We define a ``derived'' action of $T$ on $S$ by
\[x+\theta y=x+y+\cdots+y.\]
Again, it has a right-sided zero $0$. The following lemma will be our main technical tool.

\begin{lemma}\label{l:difference}
Let $f\colon S\to G$ be an arbitrary mapping of $S$ into an abelian group $G$. Then, for each prime power $\theq=p^\them$ and $\ell_0>0$, there exists $\ell\geq\ell_0$, a function $D_{q,\ell}f\colon S\times T^{\ell}\to G$ such that $D_{\theq,\ell}f(x;y_1,\ldots,y_\ell)=0$ whenever $y_i=0$ for some $i$, and	$\theta>0$ such that
\begin{flalign*}
& & f(x+\theta y) & \equiv f(x)+D_{\theq,\ell}f(x;y,\ldots,y). & (\operatorname{mod}\theq)
\end{flalign*}
In fact, $D_{\theq,\ell}$ is a formal expression in terms of $f$, the action of $T$ on $S$ and the group structure on $G$ and works universally for all $f\colon S\to G$. Moreover, this expresion is computable.
\end{lemma}

We will make a heavy use of higher-order differences
\[
	\Delta_\ell f(x;y_1,\ldots,y_\ell)=\sum_{\substack{%
		0\leq\thek\leq\ell \\
		1\leq i_1<\cdots<i_\thek\leq\ell
	}}(-1)^{\ell-\thek}f(x+y_{i_1}+\cdots+y_{i_\thek}).
\]
Clearly, $\Delta_\ell f(x;y_1,\ldots,y_\ell)=0$ whenever $y_i=0$ for some $i$.

For any formal expression written in terms of the action of $T$ on $S$, we will use a superscript $(-)^{(\theta)}$ to denote the expression obtained by replacing each $x+y$ by $x+\theta y$. In this way, we yield $\Delta_\ell^{(\theta)}f$. The function $D_{\theq,\ell}f$ will be an integral combination of the $\Delta_\ell^{(\theta)}f$.

\begin{proof}
We let $\ell=p^{\themz}$ be any power of $p$ for which $\ell\geq\ell_0$ and $\theta=p^{\themz+\them-1}$. The proof is executed by induction with respect to $\them$. By definition, $f(x+p^{\themz+\them-1}y)$ equals
\[f(x+p^\themz p^{\them-1}y)=\Delta_\ell^{(p^{\them-1})}f(x;y,\ldots,y)-\sum_{j=0}^{p^\themz-1}(-1)^{p^\themz-j}\binom{p^\themz}{j}f(x+jp^{\them-1}y).\]
For $j>0$, write $j=p^\themo j'$ where $j'$ is prime to $p$ and observe that
\[j\binom{p^\themz}{j}=p^\themz\binom{p^\themz-1}{j-1}\]
is divisible by $p^\themz$, so that $p^{\themz-\themo}\mid\binom{p^\themz}{j}$. Setting $\themo+\them=\themz+\themp$, we have either $\themp\leq 0$, in which case $\themz-\themo\geq\them$ and the binomial coefficient is divisible by $\theq=p^\them$, or we obtain for $\theqp=p^\themp$ by induction
\begin{flalign*}
& & f(x+jp^{\them-1}y)=f(x+p^{\themz+\themp-1}j'y)\equiv f(x)+D^{(j')}_{\theqp,\ell}f(x;y,\ldots,y) & & (\operatorname{mod}\theqp)
\end{flalign*}
(this holds even for $j=0$ when the last term is interpreted as $0$). Upon multiplication by $\binom{p^\themz}{j}$, that is divisible by $p^{\themz-\themo}=q/q'$, we obtain even
\begin{flalign*}
& & \binom{p^\themz}{j}f(x+jp^{\them-1}y)\equiv\binom{p^\themz}{j}f(x)+\binom{p^\themz}{j}D^{(j')}_{\theqp,\ell}f(x;y,\ldots,y). & & (\operatorname{mod}\theq)
\end{flalign*}
Since $\sum_{j=0}^{p^\themz-1}(-1)^{p^\themz-j}\binom{p^\themz}{j}=-1$, substituting the previous equation into the first yields
\begin{flalign*}
& & f(x+p^{\themz+\them-1}y) & \equiv \makebox[270pt][l]{$\displaystyle f(x)+\Delta_\ell^{(p^{\them-1})}f(x;y,\ldots,y)-\sum_{j=0}^{p^\themz-1}(-1)^{p^\themz-j}\binom{p^\themz}{j}D^{(j')}_{\theqp,\ell}f(x;y,\ldots,y)$} & \\
& & & =f(x)+D_{\theq,\ell}f(x;y,\ldots,y), & (\operatorname{mod}\theq)
\end{flalign*}
where we set $D_{\theq,\ell}=\Delta_\ell^{(p^{\them-1})}-\sum_{j=0}^{p^\themz-1}(-1)^{p^\themz-j}\binom{p^\themz}{j}D^{(j')}_{\theqp,\ell}$.
\end{proof}

\begin{example}
In this example, we have $\theq=p^\them=4$ and $\ell=4$. Then
\[f(x+8y)=\Delta_4^{(2)}f(x;y,y,y,y)+4f(x+6y)-6f(x+4y)+4f(x+2y)-f(x)\]
and we continue in a similar way with the third term,
\[f(x+4y)=\Delta_4^{(1)}f(x;y,y,y,y)+4f(x+3y)-6f(x+2y)+4f(x+y)-f(x).\]
Substituting into the first equation, we get
\begin{flalign*}
& & f(x+8y)\equiv f(x)+\Delta_4^{(2)}f(x;y,y,y,y)+2\Delta_4^{(1)}f(x;y,y,y,y) & & (\operatorname{mod}4)
\end{flalign*}
and $D_{4,4}f=\Delta_4^{(2)}f+2\Delta_4^{(1)}f$.
\end{example}

\section{Postnikov tower}

We assume that $Y$ is $\thed$-connected simplicial set and has all homotopy groups $\pi_\then Y$ finite for $2\thed<\then<\dim X=\theD$. In the following theorem, $K(\pi,n+1)$ is the Eilenberg-MacLane space and $E(\pi,n)$ its path space; more precisely, we use the canonical minimal models with both simplicial sets minimal and the projection $\delta\colon E(\pi,n)\to K(\pi,n+1)$ a minimal fibration, see~\cite{May}.

\begin{theorem}
For each simply connected simplicial set $Y$, it is possible to construct simplicial sets $P_n$, for $n<\theD$, and a sequence of simplicial maps
\[Y\xlra{\varphi_n}P_n\]
such that $\varphi_{n*}\colon\pi_i(Y)\to\pi_i(P_n)$ is an isomorphism for $i \leq n$ and $\pi_i(P_n)=0$ for $i>n$.

Further, for $2\thed<n<\theD$, it is possible to construct simplicial sets $P_{n,i}$ that fit into a pullback square
\[\xymatrix{
P_{n,i} \ar[r] \ar[d] \pb & E(\cyclic{\qni{n}{i}},n) \ar[d]^-\delta \\
P_{n,i-1} \ar[r]^-{\kni{n}{i}} & K(\cyclic{\qni{n}{i}},n+1)
}\]
with $\qni{n}{i}=\pni{n}{i}^{\mni{n}{i}}$ a prime power (depending on $n$ and $i$; the same applies to $\kni{n}{i}$) and $P_{n-1}=P_{n,0}$, $P_n=P_{n,r}$, where $r$ is some integer that depends on $n$. The composition of the projections $P_{n,i} \to P_{n,i-1}$ for $i=1,\ldots,r$ is a map $p_n \colon P_n \to P_{n-1}$ for which $p_n \varphi_n =\varphi_{n-1}$.
\end{theorem}

\begin{proof}
The paper~\cite{polypost} gives the simplicial sets $P_n$. To obtain their refinements $P_{n,i}$, we compute a decomposition
\[\pi_n\cong\cyclic{q_1}\oplus\cdots\oplus\cyclic{q_r}\]
of the $n$-th homotopy group into a sum of cyclic groups of prime power orders. Then we define $\pi_{n,i}=\cyclic{q_1}\oplus\cdots\oplus\cyclic{q_i}$ with obvious projections $\pr\colon\pi_n\to\pi_{n,i}$; $P_{n,i}$ is the following pullback
\[\xymatrixb
	{P_{n,i} \ar[rr] \ar[d] \pb & & E(\pi_{n,i},n) \ar[d]^-\delta}
	{P_{n-1} \ar[r]^-{\kn} & K(\pi_n,n+1) \ar[r]^-{\pr_*} & K(\pi_{n,i},n+1)}
\qedhere\] 
\end{proof}

\begin{theorem}\label{t:action}
It is possible to construct an action $x+\Theta y$, $\Theta\gg 0$, of $P_{2\thed}$ on each $P_{n,i}$, for $2\thed \leq n < D$, that has a right-sided zero $0\in P_{2\thed}$. The projections $P_{n,i}\to P_{n,i-1}$ respect this action.
\end{theorem}

\begin{proof}
We will construct, by induction with respect to $n$ and $i$, positive integers $\Theta_{n,i}$ and an action $x+\Theta_{n,i}y$ of $P_{2\thed}$ on $P_{n,i}$. The action $x+\Theta y$ from the statement is then obtained by setting $\Theta=\Theta_{\theD-1,r}$ and deriving the action $\Theta_{n,i}$; this is possible since $\Theta_{n,i} \mid \Theta$ by construction. Starting with $n=2\thed$, the paper~\cite{stable} constructs an abelian H-group structure on $P_{2\thed}$, i.e.\ an action of $P_{2\thed}$ on itself; we set $\Theta_{2\thed+1,0}=1$.

For the induction step, we apply Lemma~\ref{l:difference} to the Postnikov invariant $\kni{n}{i}\colon P_{n,i-1}\to K(\cyclic{\qni{n}{i}},\thenpo)$ -- its target is a simplicial abelian group, i.e.\ an abelian group in each dimension. The function
\[D_{\qni{n}{i},\ell}^{(\Theta_{n,i-1})}\kni{n}{i} \colon P_{n,i-1} \times P_{2\thed} \times \cdots \times P_{2\thed} \to K(\cyclic{\qni{n}{i}},\thenpo)\]
(formally, it is not derived from $D_{\qni{n}{i},\ell}\kni{n}{i}$ since $x+y$ is not defined, but we want to emphasize that it is with respect to the action $x+\Theta_{n,i-1}y$) is zero whenever at least one of the components in $P_{2\thed}$ is zero and thus we have a diagram
\[\xymatrix@C=40pt{
P_{n,i-1}\times\{\textrm{fat wedge}\} \ar[r]^-0 \ar@{ >->}[d] & E(\cyclic{\qni{n}{i}},\then) \ar@{->>}[d]^-\delta \\
P_{n,i-1}\times P_{2\thed}\times\cdots\times P_{2\thed} \ar[r]_-{D_{\qni{n}{i},\ell}^{(\Theta_{n,i-1})}\kni{n}{i}} \ar@{-->}[ru]^-{M'} & K(\cyclic{\qni{n}{i}},\thenpo)
}\]
(the fat wedge consists of those $\ell$-tuples $(y_1,\ldots,y_\ell) \in P_{2\thed}\times\cdots\times P_{2\thed}$ with at least one $y_i$ equal to the basepoint $0$). The cofibre of the map on the left is $(P_{n,i-1})_+\wedge P_{2\thed}\wedge\cdots\wedge P_{2\thed}$ and is $(\ell(\thed+1)-1)$-connected. Therefore, when $\ell\gg 0$, a diagonal $M'$ exists; it can be computed as in~\cite{stable}. We define $M(x,y)=M'(x;y,\ldots,y)$, so that
\[\delta M(x,y)=D_{\qni{n}{i},\ell}^{(\Theta_{n,i-1})}\kni{n}{i}(x;y,\ldots,y)=\kni{n}{i}(x+\theta\Theta_{n,i-1}y)-\kni{n}{i}(x),\]
where $\theta$ is the output of Lemma~\ref{l:difference}. Denoting $\Theta_{n,i}=\theta\Theta_{n,i-1}$, this allows us to define a new action on $P_{n,i}\subseteq P_{n,i-1}\times E(\cyclic{\qni{n}{i}},\then)$ by the formula
\[(x,c)+\Theta_{n,i}y=(x+\Theta_{n,i}y,c+M(x,y))\]
(the compatibility holds since $\delta(c+M(x,y))=\delta c+\delta M(x,y)=\kni{n}{i}(x)+(\kni{n}{i}(x+\Theta_{n,i}y)-\kni{n}{i}(x))=\kni{n}{i}(x+\Theta_{n,i}y)$).
\end{proof}

After the following simple observation, we will be ready to prove Theorem~\ref{t:extend}.

\begin{lemma}\label{l:extensions}
For each $g'\colon X\to P_{2d}$ and $2\thed<n<\theD$, it is possible to compute the finite set of homotopy classes of all lifts $g\colon X\to \Pn$.
\end{lemma}

\begin{proof}
This follows from the fact that each $\pi_n$ is finite for $2\thed<n<\theD$. Namely, since $\pi_{2\thed+1}$ is finite, the number of all lifts of $g'$ to a map $X \to P_{2\thed+1}$ is finite. Thus, it is possible to go through all these partial lifts and compute all their lifts to $P_n$ by recursion.
\end{proof}

\section{Proof of Theorem~\ref{t:extend}}

For $\then=\theD-1$, let $f\colon A\to\Pn$ also denote the composition $A \xra{f} Y \xra{\varphi_n} \Pn$. By the usual obstruction theory, it is enough to check whether an extension to $g\colon X\to\Pn$ exists -- the higher obstructions are all zero. Thus, we consider the Postnikov stage $\Pn$ with an action $x+\Theta y$ by the stage $P_{2\thed}$. Consider the commutative square (the $R$ and $R'$ are the restriction maps while $\Pi_X$ and $\Pi_A$ are post-compositions with the projection $\Pn \to P_{2\thed}$)
\[\xymatrix{
	\rightbox{[g]\in{}}{[X,\Pn]} \ar[r]^-{\Pi_X} \ar[d]_-{R} & [X,P_{2\thed}] \ar[d]^-{R'} \\
	\rightbox{[f]\in{}}{[A,\Pn]} \ar[r]_-{\Pi_A} & \leftbox{[A,P_{2\thed}]}{{}\ni[f']}
}\]
with $[f']=\Pi_A [f]$. We compute the groups on the right explicitly as in~\cite{stable} and consider the subset $H=(R')^{-1}[f']$ of all possible extensions of $f'$ to a map $X\to P_{2\thed}$. There is a finite set $H_0\subseteq H$ such that $H=H_0+\Theta\ker R'$; namely, if $[h_0]\in H$ and we identify $\ker R'\cong\cyclic{q_1}\oplus\cdots\oplus\cyclic{q_r}$ (possibly with some $q_i=0$ giving $\cyclic{0}=\bbZ$), we may take for $H_0$ all $r$-tuples of the form $[h_0]+(z_1,\ldots,z_r)\in H$ with each $|z_i|\leq\Theta/2$.

Suppose first that $g$ is any extension of $f$ and express its image in $[X,P_{2\thed}]$ as $\Pi_X[g]=[h]-\Theta [k]$ with $[h]\in H_0$ and $[k]\in\ker R'$. Then $[\widehat g]=[g]+\Theta [k]\in \Pi_X^{-1}(H_0)$ also gives an extension of $f$ since
\[
	R[\widehat g]=R([g]+\Theta [k])=[f]+\Theta R'[k]=[f],
\]
(the operations in homotopy classes are natural and $[k]\in\ker R'$). Thus, we see that an extension $g$ exists if and only if $[f]\in R\Pi_X^{-1}(H_0)$. This set is finite and its representatives can be computed using Lemma~\ref{l:extensions}. For each $[\widehat f] \in R\Pi_X^{-1}(H_0)$, we may then test whether $[\widehat f]=[f]$ by the main theorem of~\cite{homotopy}.
\qed

\section{A fibrewise equivariant version}\label{s:fibrewise}

The same argument could be repeated in the fibrewise equivariant setup of~\cite{aslep}, though actions with a strict right-sided zero have to be replaced by ones with a weak zero. Denoting $I=\stdsimp 1$, this structure is a map
\[(1 \times P_{n,i} \times_B P_{2\thed}) \cup (I \times P_{n,i} \times_B B) \to P_{n,i}\]
consisting of an action and a homotopy $x \sim x+0$.

The most significant difference lies in the proof of Theorem~\ref{t:action}. The space $P_{n,i-1} \times P_{2\thed} \times \cdots \times P_{2\thed}$ has to be replaced by the following subspace of $I^\ell \times (P_{n,i-1} \times_B P_{2\thed} \times_B \cdots \times_B P_{2\thed})$:
\begin{equation}\label{eq:weak_product}
\bigcup_{\substack{k \geq 0,\\ 1 \leq i_1 < \cdots < i_k \leq \ell}} (d_{i_1}^+\cdots d_{i_k}^+ I^\ell) \times (P_{n,i-1} \times_B \vee_B^{i_1,\ldots,i_k}P_{2\thed}),
\end{equation}
where $d_{i_1}^+\cdots d_{i_k}^+ I^\ell \subseteq I^\ell$ consists of those $\ell$-tuples $(t_1,\ldots,t_\ell)$ with $t_{i_1}=\cdots=t_{i_k}=1$ and where $\vee_B^{i_1,\ldots,i_k}P_{2\thed} \subseteq P_{2\thed} \times_B \cdots \times_B P_{2\thed}$ is formed by those $\ell$-tuples $(y_1,\ldots,y_\ell)$ whose components $y_j$ with $j \notin \{i_1,\ldots,i_k\}$ lie on the zero section $B$. In particular, $\vee_B^\emptyset P_{2\thed} = B \times_B \cdots \times_B B$ and $\vee_B^{1,\ldots,\ell} P_{2\thed} = P_{2\thed} \times_B \cdots \times_B P_{2\thed}$.

The subspace $P_{n,i-1} \times \{\textrm{fat wedge}\}$ is replaced by the subspace of \eqref{eq:weak_product} formed by those elements whose component in $I^\ell$ has at least one component equal to $0$. By the methods of~\cite{aslep}, it is then easy to equip this pair with effective homology, compute the variation of the map $M'$ from the proof of Theorem~\ref{t:action} and use it to define a new weak action of $P_{2\thed}$ on $P_{n,i}$.

\vskip 20pt
\vfill
\vbox{\footnotesize%
\noindent\begin{minipage}[t]{0.45\textwidth}
{\scshape
Luk\'a\v{s} Vok\v{r}\'inek}
\vskip 2pt
Department of Mathematics and Statistics,\\
Masaryk University,\\
Kotl\'a\v{r}sk\'a~2, 611~37~Brno,\\
Czech Republic
\vskip 2pt
\url{koren@math.muni.cz}
\end{minipage}
}

\end{document}